\documentclass[A4paper,11pt,english]{article}
\usepackage{babel}
\usepackage{graphicx}
\usepackage{amsmath}
\usepackage[utf8]{inputenc}
\usepackage{amsfonts}
\usepackage{amssymb}
\usepackage{amsmath}
\usepackage{amsthm}
\usepackage{xcolor}
\usepackage{array}
\usepackage{cite}
\usepackage{hyperref}
\title{\bf $L^{p}$-Approximation and Shape-preserving Properties of the Max-product Generalized Sampling Operators}
\author{ {\bf Lorenzo Boccali$^{1,2}$} \hskip1cm {\bf Gianluca Vinti$^{1}$} \\ 
	$^{1}$Department of Mathematics and Computer Science \\
	University of Perugia\\
	1, Via Vanvitelli, 06123 Perugia, Italy \\ 
	$^{2}$Department of Mathematics and Computer Science ``Ulisse Dini" \\ University of Florence \\ 67/a, Viale Morgagni, 50134 Florence, Italy \\
	{\small {\tt lorenzo.boccali@unifi.it}} -  {\small {\tt gianluca.vinti@unipg.it}} }
\date{}
\newtheorem{prop}{Proposition}[section]
\newtheorem{definizione}{Definition}[section]
\newtheorem{lemma}{Lemma}[section]
\newtheorem{teorema}{Theorem}[section]

\theoremstyle{definition}
\newtheorem{remark}{Remark}[section]
\newcommand{\X}{\chi}

\newcommand{\somma}{\sum_{k=-n}^{n}}
\newcommand{\N}{\mathbb{N}}
\newcommand{\R}{\mathbb{R}}
\newcommand{\Z}{\mathbb{Z}}
\newcommand{\assolutol}{\lvert}
\newcommand{\assolutor}{\rvert}
\newcommand{\norma}{\|}
\newcommand{\op}{S_{n}^{\chi}}
\newcommand{\aconchi}{a_{\chi}}

\newcommand{\integrale}{\int_{-1}^{1}}
\newcommand{\V}{\bigvee_{k=-n}^{n}}

\begin{document}
	\maketitle
\begin{abstract}
In this paper, we investigate the convergence in the $L^{p}$-norm and certain shape-preserving properties of the max-product generalized sampling operators. More precisely, we establish quantitative estimates for the approximation error in the $L^{p}$-norm, for $ 1 \le p < +\infty$, in the case of non-negative and bounded functions defined on $[-1,1]$. These estimates are derived by means of the so-called $\tau$-modulus, an averaged modulus of smoothness introduced by Sendov and Popov. As a direct consequence, we prove that the max-product generalized sampling operators $L^{p}$-converge to non-negative functions that are measurable, bounded and Riemann integrable on the interval $[-1,1]$. In the final section, we extend several shape-preserving results of Coroianu and Gal, originally established for specific kernels (such as the sinc/Whittaker and Fejér kernels), to the broader class of smooth centered bell-shaped kernels. Under suitable assumptions on the kernel, we prove that the max-product generalized sampling operators partially preserve the monotonicity of any function $f:[0,1] \rightarrow \R_{0}^{+}$ that is either non-decreasing or non-increasing on $[0,1]$.
\end{abstract}
\medskip\noindent
{\small {\bf AMS subject classification:} 41A25, 41A05  \newline
{\small {\bf Key Words:} Max-product generalized sampling operators, Quantitative estimates, Averaged modulus of smoothness, shape-preserving properties, centered bell-shaped functions

\section{Introduction}
Over the past few years, within the field of Operator Theory, the max-product approach applied to several families of linear operators has become a research topic of increasing interest (see, e.g., \cite{kadak2022max,coroianu2022approximation,bajpeyi2024exponential}). For a comprehensive overview, see the monograph \cite{bede2016approximation}. The main motivation behind this particular focus lies in the improved approximation order obtained by replacing, in several families of discrete linear operators, the series (or sum, in the case of finite terms) with the supremum (or maximum, in the finite case), denoted by the symbol $\bigvee$, of a suitable set of real numbers. \newline Recently, the max-product version of sampling operators based upon a general kernel function $\chi$ (satisfying suitable assumptions) was introduced by Coroianu et al. \cite{coroianu2019max} and it is defined by:
\begin{equation*}
S_{n}^{\chi}(f)(x):= \frac{\displaystyle \bigvee_{k=\lceil na \rceil}^{\lfloor nb \rfloor} f\left(\frac{k}{n}\right) \chi(nx-k)}{\displaystyle \bigvee_{k=\lceil na \rceil}^{\lfloor nb \rfloor} \chi(nx-k)}, \ \ x \in [a,b], \ n \in \N^{+}, \quad \quad \quad \quad \textnormal{(I)}
\end{equation*}
for any $f:[a,b] \rightarrow \R$ bounded function. Here, the symbols $\lceil \cdot \rceil$ and $\lfloor \cdot \rfloor$ denote, respectively, the ``ceiling" and the ``integer part" of a given number. \newline The non-linear (sub-additive) operators (I) represent one possible non-linear generalization of the celebrated sampling series \cite{butzer1987approximation} introduced in the 1980s by the German mathematician P.L. Butzer. The goal of Butzer's theory was to provide an approximate version of the classical Whittaker-Kotel'nikov-Shannon (WKS) sampling theorem (see, e.g., \cite{butzer1981shannon,butzer2001sampling}), pioneer result of Sampling Theory, that would be more useful from an applied point of view. On the one hand, max-product generalized sampling operators preserve the same approximation capability for continuous functions as their corresponding linear counterpart. On the other hand, they offer the advantage of providing convergence results even when based on kernels that satisfy weaker assumptions than those in Butzer's theory, namely, kernels that are not-necessarily approximate identities (see, e.g., \cite{butzer1988sampling}). \newline Following a moment-type approach, based on defining the notion of generalized absolute moment of a given kernel as the max-product version of the usual discrete absolute moment typically used in linear approximation problems (see, e.g., \cite{coroianu2024approximation,costarelli2025regularization,costarelli2023convergence2,ries1984approximation,costarelli2023alzheimer}), the authors in \cite{coroianu2019max} obtained pointwise and uniform convergence, along with quantitative estimates, in the space of (non-negative) continuous functions $f:[a,b] \rightarrow \R_{0}^{+}$, or, respectively, in the space of uniformly continuous and bounded functions $f: \R \rightarrow \R_{0}^{+}$.  
\newline  However, the study of the approximation properties for the operators $S_{n}^{\chi}$ with respect to the $L^{p}$-norm remains an uncovered topic. Based on this consideration, in the first part of the present paper, we investigate the power of the max-product operators (I) in approximating (non-negative) bounded functions belonging to $L^{p}$-spaces, for $1 \le p < +\infty$. However, this choice introduces some technical problems. A primary difficulty arises from their own definition: since the evaluation of the signal $f$ occurs at the nodes $k/n$, for $k=\lceil na \rceil, ..., \lfloor nb \rfloor$, $n \in \N^{+}$, the operators considered in this paper are less suitable for reconstructing functions that are not-necessarily continuous than their Kantorovich variant \cite{coroianu2021approximation}. This issue stems from the fact that any family of discrete operators that depend on the pointwise behavior of the function may map different functions belonging to the same equivalence class to distinct classes. Consequently, the max-product generalized sampling operators $S_{n}^{\chi}f$ do not necessarily converge to $f$ with respect to the $L^{p}$-norm. Moreover, the usual $L^{p}$-setting is not suitable for studying pointwise-dependent operators, since the identification of functions via equivalence classes becomes useless in these cases. To overcome this limitation for the theory, we adopt in this paper a not-usual definition of the $L^{p}$-setting, widely used in the literature when dealing with approximation processes where one directly deals with the pointwise values of the function at a finite numbers of points (see, e.g., \cite{costarelli2023approximation3,bardaro2006approximation}). \newline Within this framework, we provide quantitative estimates for the approximation error with respect to the usual norm $\| \cdot \|_{p}$ by exploiting the so-called $\tau$-modulus, an averaged modulus of smoothness introduced by the Bulgarian school under Sendov \cite{sendov1988averaged}. This tool, defined for bounded functions, has been extensively used in Approximation Theory to derive $L^{p}$-approximation error estimates for several pointwise-dependent operators, such as the classical (linear) Bernstein, Sz\'asz-Mirakjan and Baskakov operators (see \cite{sendov1988averaged}, Chapter 4). To achieve the desired purpose, after recalling in Section \ref{Sezione2} the definition of the max-product operators (I) and the known results used in Section \ref{Sezione3}, we first prove a Lipschitz property with respect to $\| \cdot \|_{p}$ for the approximation of bounded, measurable, and non-negative functions. Then, using this preliminary result along with a quantitative estimate of the approximation error, established for (non-negative) functions belonging to Sobolev space $W^{1}_{p}([-1,1])$, we are able to extend the same estimate to any non-negative and bounded function $f \in L^{p}([-1,1])$ by density approach. For $p>1$, the proof of the estimate in the Sobolev setting is based on the well-known Hardy-Littlewood maximal function (see, e.g., \cite{aldaz2012optimal}). On the other hand, for $p=1$, we assume that the kernels are compactly supported to obtain an analogous result. In both cases, the main estimate, combined with the well-known properties of the $\tau$-modulus, ensures that the family $(S_{n}^{\chi}(f))_{n \in \N^{+}}$ converges to $f$ in  $L^{p}([-1,1])$ whenever $f$ is measurable, bounded, non-negative, and Riemann integrable on the interval $[-1,1]$.  
\newline In addition to providing good approximation capabilities, max-product operators generally exhibit interesting shape-preserving properties (see, e.g., \cite{bede2009approximation,bede2010approximation,mansoori2025approximation}). In \cite{coroianu2011approximation}, the authors proved that the max-product versions of the truncated Whittaker cardinal series and the truncated sampling operator based upon the Fejér kernel partially preserve monotonicity. The goal of Section \ref{Sezione4} is to find, following the same approach, a broader class of kernels that conserve the same shape-preserving properties. In this paper, we show that it is sufficient to consider kernels satisfying the notion of a (smooth) centered bell-shaped function, as introduced by Cardaliaguet and Euvrard in \cite{cardaliaguet1992approximation}. More precisely, under additional assumptions on the kernel function, we prove that the family $(S_{n}^{\chi}(f))_{n \in \N^{+}}$ partially preserves monotonicity for any bounded function $f:[0,1] \rightarrow \R_{0}^{+}$ that is either non-decreasing or non-increasing on $[0,1]$. \newline At the end of the paper, we present some examples of kernel functions to which the above theory applies.  

\section{Preliminaries, notations and known results \label{Sezione2}}
We begin this section by recalling a useful definition, which is crucial for defining the kind of sampling operators studied in this paper. Following the usual notation in the literature (see, e.g., \cite{yuksel2018approximation,anastassiou2018nonlinearity,boccali2024max,costarelli2019convergence,bede2016approximation}), we denote by $\bigvee$ the max-product symbol, defined as follows:   
\begin{equation}
\label{max-productsymbol}
\bigvee_{k=K_{1}}^{K_{2}}A_{k}:=\max\{A_{k} \in \R: k=K_{1},...,K_{2}\},
\end{equation}
for any pair of integers $K_{1} \le K_{2}$. Similarly, we define $\bigvee_{k \in \Z}$ by replacing the maximum with the supremum, and $k=K_{1},...,K_{2}$ with $k \in \Z$ on the right-hand side of (\ref{max-productsymbol}).  \newline Moreover, in order to introduce the definition of max-product generalized sampling operators, we recall the notion of kernel function introduced in \cite{coroianu2019max} to establish qualitative and quantitative pointwise and uniform convergence properties in spaces of continuous functions. From now on, we consider kernels of the form $\chi: \R \rightarrow \R$, where $\chi$ is any bounded and measurable function satisfying the following properties: 
\vskip0.2cm $(\chi1)$ For a suitable $\beta>0$, there holds:
\begin{equation*}
	m_{\beta}(\chi):=\sup_{x \in \R} \bigvee_{k \in \Z} \lvert \chi(x-k) \rvert \cdot \lvert x-k \rvert^{\beta} = \sup_{u \in \R} \lvert \chi(u) \rvert \lvert u \rvert^{\beta} < +\infty, 
\end{equation*}
\hskip0.6cm i.e., the generalized absolute moment of order $\beta$ of $\chi$ is finite;
\vskip0.05cm $(\chi2)$ $\chi(x)$ is bounded on $\left[-1, 1\right]$ with strictly positive lower bound, i.e.,
\begin{equation*}
	\inf_{x \in \left[-1, 1 \right]} \chi(x)=:a_{\chi}>0.
\end{equation*}
Note that, in this context, we do not require the kernel $\chi$ to satisfy the usual assumptions of discrete approximate identities, which are instead necessary according in Butzer's theory (see, e.g., \cite{butzer1987approximation,butzer1990generalized}).
\newline In what follows, we summarize some known results that relate the kernel function $\chi$ to the max-product symbol $\bigvee$.
\begin{lemma}
\label{condizionesufficientemomenti}
If $\chi:\R \rightarrow \R$ is bounded and such that $\chi(x)=\mathcal{O}(\lvert x \rvert^{-\alpha})$, as $\assolutol x \assolutor \rightarrow +\infty$, for $\alpha >0$, then:
\begin{equation*}
m_{\beta}(\chi)<+\infty, \ for \ every \ \ 0 \le \beta \le \alpha.
\end{equation*}
\end{lemma}
Lemma \ref{condizionesufficientemomenti} provides a sufficient condition for the existence of generalized absolute moments. Its proof can be easily deduced by replacing the sum (or series) with the symbol $\bigvee$ in a lemma originally proved in \cite{ries1984approximation} concerning classical discrete absolute moments. 
\begin{lemma}[see Lemma 2.4 of \cite{costarelli2016max}]
\label{lemma1.2}
Let $\chi:\R \rightarrow \R$ be a bounded function satisfying $(\chi1)$ with $\beta >0$. Then: 
\begin{equation*}
m_{\upsilon}(\chi) < +\infty, \ for \ every \ \  0 \le \upsilon \le \beta.
\end{equation*}
In particular, it turns out that $m_{0}(\chi) \le \| \chi \|_{\infty}$, where, from now on, $\| \cdot \|_{\infty}$ denotes the usual sup-norm. 
\end{lemma}
\begin{lemma}
\label{disuguaglianzacruciale}
Let $\chi: \R \rightarrow \R$ be a given function satisfying $(\chi2)$ and let $a<b$ be fixed. Then for every $x \in [a, b] \subset \R$ and $n \in \N^{+}$, we have:
\begin{equation}
\label{disuguaglianzaaconchi}
\bigvee_{k=\lceil na \rceil}^{\lfloor nb \rfloor} \chi(nx-k) \ge a_{\chi}>0,
\end{equation}
where, from now on, $\lceil \cdot \rceil$ and $\lfloor \cdot \rfloor$ denote respectively the  ``ceiling" and the  ``integer part" of a given number, and $a_{\chi}$ is the constant arising from condition $(\chi2)$. 
\end{lemma}
\begin{proof}
Let $x \in [a,b]$ and $n \in \N^{+}$ be fixed. Since there exists at least a $\overline{k}=\lceil na \rceil,...,\lfloor nb \rfloor$ such that $\assolutol nx - \overline{k} \assolutor \le 1$, by using $(\chi2)$, we immediately obtain:
\begin{equation*}
\bigvee_{k=\lceil na \rceil}^{\lfloor nb \rfloor} \chi(nx-k) \ge \chi(nx-\overline{k}) \ge a_{\chi}>0.
\end{equation*}
This completes the proof. 
\end{proof}
The inequality established in Lemma \ref{disuguaglianzacruciale} plays a crucial role in the well-definition of the family of sampling operators studied in the following sections, as becomes evident from:
\begin{definizione}
\label{definizione1.1}
Let $f: [a,b] \rightarrow \R$ be a bounded function and let $\chi$ be a kernel satisfying properties $(\chi1)$ and $(\chi2)$. Then the max-product generalized sampling operators acting on $f$ and based upon $\chi$ are defined by:
\begin{equation}
\label{definizioneoperatori}
S_{n}^{\chi}(f)(x):= \frac{\displaystyle \bigvee_{k=\lceil na \rceil}^{\lfloor nb \rfloor} f\left(\frac{k}{n}\right) \chi(nx-k)}{\displaystyle \bigvee_{k=\lceil na \rceil}^{\lfloor nb \rfloor} \chi(nx-k)}, \ \ x \in [a,b], 
\end{equation} 
for any $n \in \N^{+}$. 
\end{definizione}
Note that, since Lemma \ref{disuguaglianzacruciale} holds, the denominator in the right-hand side of (\ref{definizioneoperatori}) is always strictly positive. Moreover, since $f$ is bounded, we have: 
\begin{equation*}
\assolutol S_{n}^{\chi}(f)(x) \assolutor \le  \frac{\norma f \norma_{\infty}}{a_{\chi}} m_{0}(\chi) < +\infty, \ \ x \in [a,b], 
\end{equation*}
where the zero-order generalized absolute moment is finite in view of Lemma \ref{lemma1.2}. Therefore, the above operators are well-defined. \newline The max-product operators in (\ref{definizioneoperatori}) originally appeared in Coroianu et al. \cite{coroianu2010approximation}, where the authors considered only compactly supported kernels or specific cases (see also \cite{coroianu2011approximation}), such as sinc and Fejér kernels. Definition \ref{definizione1.1}, which extends the theory to any general kernel function (not-necessarily compactly supported) satisfying the (weak) assumptions $(\chi1)$ and $(\chi2)$, was later introduced in \cite{coroianu2019max}. In that paper, pointwise and uniform convergence for the operators $S_{n}^{\chi}$ was established for functions $f:[a,b] \rightarrow \R^{+}_{0}$ or $f:\R \rightarrow \R^{+}_{0}$, which are continuous or uniformly continuous and bounded, respectively. In many cases (see, e.g., \cite{coroianu2011classes,bede2016approximation,coroianu2011approximation}), the max-product version of a family of linear operators achieves a finer approximation, expressed by a better order of approximation than its linear counterpart. This advantage also applies to generalized sampling operators, for which a Jackson-type estimate in terms of the modulus of continuity of the approximated function has been proved in \cite{coroianu2019max}. \newline Moreover, when dealing with max-product-type operators, it is common to establish certain properties that will be particularly useful in the next section for analyzing the approximation power of the operators $S_{n}^{\chi}$ in $L^{p}$-spaces, with $1 \le p < +\infty$. We summarize these properties in the following lemma. For a proof, see Lemma 3.2 of \cite{coroianu2019max}.
\begin{lemma}
\label{lemma1.4}
Let $\chi$ be a kernel satisfying assumptions $(\chi1)$ and $(\chi2)$. Further, let $a<b$ be fixed. If $f,g: [a,b] \rightarrow \R_{0}^{+}$ are two bounded functions, then for any $n \in \N^{+}$, we have:
\vskip0.2cm $(i)$ If $f(x) \le g(x)$, for every $x \in [a,b]$, then $S_{n}^{\chi}(f)(x) \le S_{n}^{\chi}(g)(x)$, for each $x \in [a,b]$. 
\vskip0.1cm $(ii)$ $S_{n}^{\chi}$ is sub-additive (or sub-linear) operator, i.e., $S_{n}^{\chi}(f+g)(x) \le S_{n}^{\chi}(f)(x) + S_{n}^{\chi}(g)(x)$, for \vskip0.01cm $x \in [a,b]$.
\vskip0.1cm $(iii)$ $\assolutol S_{n}^{\chi}(f)(x)-S_{n}^{\chi}(g)(x) \assolutor \le S_{n}^{\chi}(\assolutol f-g\assolutor)(x)$, for each $x \in [a,b]$.
\vskip0.1cm $(iv)$ $S_{n}^{\chi}$ is positively homogeneous operator, i.e., $S_{n}^{\chi}(\lambda f)(x)=\lambda S_{n}^{\chi}(f)(x)$, for each $\lambda>0$ and \vskip0.01cm for all $x \in [a,b]$. 
\end{lemma}
The goal of the next section is to investigate the approximation capabilities of the max-product generalized sampling operators defined in (\ref{definizioneoperatori}) with respect to $L^{p}$-convergence. This can be achieved by following an approach commonly adopted in the literature (see, e.g., \cite{bardaro2006approximation,costarelli2023approximation3,sendov1988averaged}) when dealing with the evaluation of discrete operators in the case of functions belonging to $L^{p}$-spaces. The key idea is to deviate from the usual definition of $L^{p}$-spaces, where functions that coincide almost everywhere are identified as belonging to the same equivalence class. Instead, in this context, we separate each bounded measurable function from its own equivalence class. The main motivation behind this choice stems from the fact that the definition of the sampling operators in (\ref{definizioneoperatori}) depends on the evaluation of the signal $f$ at the nodes $k/n$, where $k=\lceil na \rceil,..., \lfloor nb \rfloor$ and $n \in \N^{+}$. Since these values form a zero-measure set, the operators $S_{n}^{\chi}$ do not generally converge to $f$ in the $L^{p}$-norm. In view of this observation, from now on, each function defined on a bounded interval $[a,b]$ is uniquely determined by its pointwise values. Therefore, we do not identify functions that coincide almost everywhere on $[a,b]$, i.e., we do not work with equivalence classes of functions. Consequently, $\|\cdot\|_{p}$ defines only a semi-norm on $L^{p}([a,b])$.     
\newline On the other hand, the pointwise nature of the max-product generalized sampling operators influences the choice of the tool used to establish approximation results in the above context. In particular, if the goal of Section \ref{Sezione3} is to estimate the approximation error for the operators $S_{n}^{\chi}$ with respect to the $L^{p}$-norm, then even for functions that are not-necessarily continuous, the classical $L^{p}$-modulus of smoothness is not the most suitable tool for evaluating the increments of $f$ on subsets of $[a,b]$ with null measure. To overcome this issue, it is preferable to use an averaged version of the modulus of smoothness, known in the literature as the first-order $\tau$-modulus, introduced by E.P. Dolzenko (1969) and E.A. Sevast'janov (1976) \cite{Dolzenko} for estimating the best piecewise monotone approximations with respect to the Hausdorff distance (see \cite{sendov1979}). In order to introduce this definition, we first recall the well-known notion of the local modulus of smoothness for functions $f \in M([a,b])$, where, from now on, $M([a,b])$ denotes the set of all bounded and measurable functions $f:[a,b] \rightarrow \R$.   
\begin{definizione}
Let $f\in M([a,b])$ be fixed. We define the local modulus of smoothness of order $r \in \N^{+}$ of the function $f$ at a point $x \in [a,b]$ as follows:
\begin{equation}
\omega_{r}(f,x;\delta):=\sup\left\{\assolutol(\Delta_{h}^{r}f)(t)\assolutor: t, t+rh \in \left[x-\frac{r\delta}{2}, x+\frac{r\delta}{2}\right] \cap [a,b]\right\},
\end{equation}
where $0 < \delta \le \frac{b-a}{r}$, and
\begin{equation*}
(\Delta_{h}^{r}f)(t):=\sum_{k=0}^{r}\binom{r}{k}(-1)^{r-k}f(t+kh).
\end{equation*} 
\end{definizione}
Considered as a function of $x \in [a,b]$, the local modulus of smoothness satisfies $\omega_{r}(f,x;\delta) \in M([a,b])$ (see Theorem 1.3 of \cite{sendov1988averaged}). It follows immediately that $\omega_{r}(f,\cdot;\delta)$ is Lebesgue integrable to the $p$-th power, i.e., $\omega_{r}(f,\cdot;\delta) \in L^{p}([a,b])$ for $1 \le p < +\infty$. This observation leads to the following definition:
\begin{definizione}
\label{definizione3.1}
Let $f \in M([a,b])$ be fixed. The averaged modulus of smoothness (or $\tau$-modulus) of order $r \in \N^{+}$ of the function $f$ is defined by:
\begin{equation}
\label{taumodulo}
\tau_{r}(f,\delta)_{p}:=\| \omega_{r}(f, \cdot;\delta)\|_{p}=\left\{\int_{a}^{b} [\omega_{r}(f,x;\delta)]^{p} \ dx\right\}^{1/p}, \ 1 \le p < +\infty,
\end{equation}
for $0 < \delta \le \frac{b-a}{r}$.
\end{definizione} 
The $\tau$-moduli were originally introduced by B. Sendov \cite{sendov1968approximations} and P.P. Korovkin \cite{korovkin} in a different but equivalent form to investigate the convergence of sequences of positive linear operators with respect to different metrics \cite{sensov1977}. Subsequently, they were formulated as in Definition \ref{definizione3.1} by the Bulgarian school led by B. Sendov \cite{sendov1988averaged}, which developed a very detailed theory of these moduli. A notable peculiarity of $\tau_{r}(f,\delta)_{p}$ is that it allows one to evaluate the increments of a given function $f$ at any fixed point $x \in [a,b]$, in contrast to the usual integral (or $L^{p}$-)moduli of smoothness, which do not capture the increments of $f$ on subsets of $[a,b]$ of measure zero.
\newline In what follows, we recall some useful properties of the $\tau$-modulus established in \cite{sendov1988averaged}:
\vskip0.2cm $(i)$ monotonicity: \[\tau_{r}(f,\delta')_{p} \le \tau_{r}(f,\delta'')_{p}, \ \textnormal{for} \ \ 0 < \delta' \le \delta'';\]
\vskip0.1cm $(ii)$ sub-additivity: 
\[\tau_{r}(f+g, \delta)_{p} \le \tau_{r}(f, \delta)_{p}+\tau_{r}(g, \delta)_{p}, \ \ \delta>0;\]  
\vskip0.1cm $(iii)$ estimate of higher order modulus by means of lower order one:
\[\tau_{r}(f,\delta)_{p} \le 2 \tau_{r-1}\left(f,\frac{r}{r-1} \ \delta\right)_{p}, \ \ \delta>0, \ r \ge 2;\]
\vskip0.1cm $(iv)$ estimate of the modulus of order $r\ge 2$ of $f$ by one of the order $r-1$ of the derivative $f'$ \vskip0.01cm (whenever it exists):
\[\tau_{r}(f,\delta)_{p} \le \delta \tau_{r-1}\left(f',\frac{r}{r-1} \ \delta\right)_{p}, \ \ \delta >0; \]
\vskip0.1cm $(v)$ inequality with respect to the product of the parameter $\delta> 0$ by positive integers:
\[\tau_{r}(f,n \delta)_{p} \le (2n)^{r+1}\tau_{r}(f, \delta)_{p}, \ \ n \in \N^{+}.\]
In the final part of this section, we provide some well-known results that highlight important properties of the $\tau$-modulus, which will play a crucial role in the next section.  We begin with the following lemma, proved in Lemma 4.3 of \cite{costarelli2023approximation3} from a slight modification of Lemma 2.5 of \cite{sendov1988averaged}.  
\begin{lemma}
\label{Lemma1.5}
Let $r \in \N^{+}$ and $1 \le p < +\infty$ be fixed. It turns out that:
\begin{equation}
\left\{\sum_{k=-n}^{n}\left[\omega_{r}\left(f, \frac{k}{n}, 2h\right)\right]^{p}n^{-1}\right\}^{1/p} \le 2^{\frac{1}{p}+2(r+1)} \tau_{r}\left(f, h +\frac{1}{(n+1)r}\right)_{p}, \ \ n \in \N^{+},
\end{equation}
for all $0 < h \le \frac{2}{r}$. 
\end{lemma}
The following property describes the behaviour of $\tau_{r}(f,\delta)_{p}$ for $\delta\rightarrow0^{+}$ under suitable assumptions on the function $f$. 
\begin{prop}[see \cite{bardaro2006approximation}, Proposition 6 and Remark 7]
\label{proposizione1.1}
Let $r \in \N^{+}$ and $f \in M([a,b])$ be fixed. If $f$ is Riemann integrable, then:
\begin{equation}
\label{cnes}
\lim_{\delta\rightarrow0^{+}} \tau_{r}(f,\delta)_{p}=0, \ \ 1 \le p < +\infty. 
\end{equation}
In the particular case of $r=1$, the vice versa holds, i.e., if $\tau_{1}(f, \delta)_{p}$ converges to zero for $\delta \rightarrow 0^{+}$, then $f \in M([a,b])$ is Riemann integrable.  
\end{prop}
Finally, a well-known density result proved by Sendov and Popov (see Theorem 2.5' of \cite{sendov1988averaged}, p. 34) establishes a connection between the $\tau$-modulus $\tau_{r}(f, \delta)_{p}$ and a particular class of functions, namely the so-called Steklov functions (see, e.g., \cite{costarelli2024convergence}). Let $W_{p}^{r}([a,b])$ denote the usual Sobolev spaces, i.e., the subspaces of $L^{p}([a,b])$ consisting of all functions $f:[a,b] \rightarrow \R$ whose $(r-1)$-th derivative is absolutely continuous and whose $r$-th derivative belongs to $L^{p}([a,b])$, with $1 \le p < +\infty$ and $r \in \N^{+}$. Then, we can state the following:
\begin{teorema}
\label{teoremaSteklov}
Let $f$ be a bounded function belonging to $L^{p}([a,b])$, for $1 \le p < +\infty$. For every $r \in \N^{+}$ and $0 < h \le \frac{b-a}{r}$, there exists a Steklov function $f_{r,h} \in L^{p}([a,b])$, defined by: 
\[f_{r,h}(x):=(-h)^{-r}\int_{0}^{h}\cdots \int_{0}^{h}\sum_{m=1}^{r}(-1)^{r-m+1}\binom{r}{m} f\left(x+\frac{m}{r}(t_{1}+...+t_{r})\right) \ dt_{1}... dt_{r},\]
where, here, $f$ has to be considered extended on $\R$, as a periodic function of period $b-a$, satisfying the following properties:
\vskip0.2cm $(i)$ $\assolutol f(x)-f_{r,h}(x) \assolutor \le \omega_{r}(f,x;2h), \ x \in [a,b]$;
\vskip0.1cm $(ii)$ $\| f-f_{r,h} \|_{p} \le \tau_{r}(f,2h)_{p}$;
\vskip0.1cm $(iii)$ $f_{r,h} \in W_{p}^{r}([a,b])$ and for its $s$-th derivative, the following inequality holds:
\[\norma f_{r,h}^{(s)} \norma_{p} \le c(r) h^{-s}\tau_{s}(f,h)_{p}, \ \ s=1, ..., r, \]
\vskip0.01cm where the constant $c(r)$ depends only on $r$.
\end{teorema}
\section{Error estimate and convergence in the $L^{p}$-norm for the operators $S_{n}^{\chi}$} \label{Sezione3}}
In this section, we establish estimates for the approximation error with respect to the $L^{p}$-norm, in fact in the case of functions $f:[a,b] \rightarrow \R^{+}_{0}$ that are not-necessarily continuous, for the max-product generalized sampling operators. For simplicity, let us consider the case of bounded (non-negative) functions defined on the interval $I:=[-1,1]$. Clearly, all the results proved in this section can be easily extended to any general interval $[a,b]$. \newline Throughout this section, unless otherwise stated, the kernel $\chi$ from Definition \ref{definizione1.1} will be any bounded and measurable function that satisfies properties $(\chi1)$ and $(\chi2)$. 
\newline In this setting, the definition of the operators $S_{n}^{\chi}$ in (\ref{definizioneoperatori}) can be reformulated as follows:
\begin{equation}
\label{operatori[-1,1]}
S_{n}^{\chi}(f)(x)=\frac{\displaystyle \bigvee_{k=-n}^{n}f\left(\frac{k}{n}\right)\chi(nx-k)}{\displaystyle \bigvee_{k=-n}^{n}\chi(nx-k)}, \ \ x \in I=[-1,1], \ n \in \N^{+},
\end{equation}
where $f:I\rightarrow\R_{0}^{+}$ is bounded. 
\vskip0.1cm 
\begin{remark}
\label{remark2.1}
Note that when considering the approximation of functions defined on the compact interval $[-1,1]$, condition $(\chi2)$ can be replaced with the following slightly weaker assumption:
\begin{equation*}
(\overline{\chi2}) \inf_{x \in \left[-\frac{1}{2}, \frac{1}{2} \right]} \chi(x)=:a_{\chi}>0.
\end{equation*}
Indeed, for all $x \in I$ and $n \in \N^{+}$, there exists at least a $\overline{k}=-n,...,n$ such that $\assolutol nx-\overline{k} \assolutor \le \frac{1}{2}$. Therefore, by reasoning as in the proof of Lemma \ref{disuguaglianzacruciale}, we obtain $\bigvee_{k=-n}^{n} \chi(nx-k) >0$ for each $x \in I$. On the other hand, all results proved in this section still hold if, instead of condition $(\chi2)$ (or $(\overline{\chi2})$) we directly assume that the kernel function $\chi:\R \rightarrow \R$ satisfies inequality (\ref{disuguaglianzaaconchi}), while keeping condition $(\chi1)$.  As a consequence of this observation, we can extend the theory to kernels that are not necessarily bounded on the interval $[-1, 1]$ (or $[-\frac{1}{2}, \frac{1}{2}]$) with a strictly positive lower bound, as in the case, for example, with asymmetric functions. 
\end{remark}
In order to achieve the goal of this section, we first need to prove some preliminary results. We begin with the following Lipschitz property in the $L^{p}$-setting for the operators $S_{n}^{\chi}$ in (\ref{operatori[-1,1]}) when evaluated on non-negative, measurable and bounded functions. 
\begin{lemma}
\label{Lipschitzproperty}
Let $\chi \in L^{1}(\R)$ be a given kernel and let $1 \le p < +\infty$ be fixed. Further, let us consider two bounded and non-negative functions $f, g: I \rightarrow \R^{+}_{0}$. Then, it turns out that:
\begin{equation}
\norma S_{n}^{\chi}(f)-S_{n}^{\chi}(g) \norma_{p} \le \frac{((m_{0}(\chi))^{p-1}\norma \chi \norma_{1})^{1/p}}{a_{\chi}} \norma f-g \norma_{l^{p}(\Sigma_{n})},
\end{equation}
where:
\begin{equation*}
\norma f-g \norma_{l^{p}(\Sigma_{n})}:=\left\{\sum_{k=-n}^{n} \bigg|  f\left(\frac{k}{n}\right)-g\left(\frac{k}{n}\right) \bigg|^{p}n^{-1}\right\}^{1/p}
\end{equation*}
denotes a discrete $l^{p}$-norm of the function $f-g$ on the partition \[\Sigma_{n}:=\left\{\frac{k}{n}, \ k=-n,...,n\right\}, \ \ n \in \N^{+},\]
of the interval $I$. Note that $m_{0}(\chi)<+\infty$ in view of Lemma \ref{lemma1.2}. 
\end{lemma}
\begin{proof}
Let $1 \le p < +\infty$ be fixed. Applying the norm $\norma \cdot \norma_{p}$ and using property $(iii)$ of Lemma \ref{lemma1.4}, we can write what follows: 
\begin{equation*}
\begin{split}
\norma\op(f)-\op(g)\norma_{p}^{p} &= \integrale \assolutol \op(f)(x)-\op(g)(x) \assolutor^{p} \ dx \\
& \le \integrale (\op(\assolutol f-g \assolutor)(x))^{p} \ dx \\
& = \integrale \left\{\frac{\displaystyle\V \left\assolutol f\left(\frac{k}{n}\right)-g\left(\frac{k}{n}\right)\right\rvert \X(nx-k)}{\displaystyle \V \X(nx-k)}\right\}^{p}  dx
\end{split}
\end{equation*}
\begin{equation*}
\begin{split}
& \le \frac{1}{\aconchi^{p}} \integrale \left\{ \V \left\assolutol f\left(\frac{k}{n}\right)-g\left(\frac{k}{n}\right)\right\assolutor \assolutol \X(nx-k) \assolutor\right\}^{p} dx \\
& = \frac{1}{a_{\chi}^{p}} \int_{-1}^{1} \left\{\bigvee_{k=-n}^{n} \left\lvert f\left(\frac{k}{n}\right)-g\left(\frac{k}{n}\right) \right\rvert^{p} \assolutol \chi(nx-k) \assolutor^{p}\right\} \ dx.
\end{split}
\end{equation*}
Now, recalling the definition of the generalized absolute moment of order zero of $\chi$, which is finite since Lemma \ref{lemma1.2} holds, we have:
\begin{equation*}
\begin{split}
\norma \op(f) - \op(g) \norma_{p}^{p} & \le  \frac{1}{a_{\chi}^{p}} \int_{-1}^{1} \left\{ \bigvee_{k=-n}^{n} \left\lvert f\left(\frac{k}{n}\right)-g\left(\frac{k}{n}\right) \right\rvert^{p} \assolutol \chi(nx-k) \assolutor^{p}\right\} \ dx \\
&\le \frac{(m_{0}(\chi))^{p-1}}{\aconchi^{p}} \int_{-1}^{1} \left\{\bigvee_{k=-n}^{n} \left\lvert f\left(\frac{k}{n}\right)-g\left(\frac{k}{n}\right)\right\rvert^{p}\assolutol \chi(nx-k) \assolutor \right\} \ dx  \\
& \le \frac{(m_{0}(\chi))^{p-1}}{a_{\chi}^{p}} \int_{-1}^{1} \left\{\sum_{k=-n}^{n} \left\lvert f\left(\frac{k}{n}\right)-g\left(\frac{k}{n}\right)\right\rvert^{p} \assolutol \chi(nx-k) \assolutor\right\} \ dx
\end{split}
\end{equation*}
\begin{equation}
\label{eq1}
\hskip2.2cm = \frac{(m_{0}(\chi))^{p-1}}{a_{\chi}^{p}}  \left\{\sum_{k=-n}^{n} \left\lvert f\left(\frac{k}{n}\right)-g\left(\frac{k}{n}\right)\right\rvert^{p} \int_{-1}^{1} \assolutol \chi(nx-k) \assolutor \ dx \right\}.
\end{equation}
For every $k= -n, ..., n$, using the substitution $y=nx-k$, we immediately get: 
\begin{equation*}
\integrale \assolutol \X(nx-k) \assolutor \ dx \le \frac{1}{n} \int_{\R} \assolutol \X(y) \assolutor \ dy = \frac{1}{n} \norma \X \norma_{1}<+\infty,
\end{equation*}
since $\chi \in L^{1}(\R)$. Finally, using this inequality in (\ref{eq1}), we obtain:
\begin{equation*}
\norma \op(f) - \op(g) \norma_{p}^{p} \le \frac{(m_{0}(\chi))^{p-1}\norma \chi \norma_{1}}{a_{\chi}^{p}} \left\{\somma \left \assolutol f\left(\frac{k}{n}\right)-g\left(\frac{k}{n}\right) \right \assolutor^{p} n^{-1} \right\},
\end{equation*}
for every $n \in \N^{+}$. This completes the proof.
\end{proof}
In the following theorems, we investigate the approximation error (in the $L^{p}$-norm) when dealing with (non-negative) functions belonging to the Sobolev space $W_{p}^{1}(I)$, with $1 \le p < +\infty$. To this aim, we need to distinguish two cases: $i)$ $1<p<+\infty$ and $ii)$ $p=1$. The reason for this distinction is related to the fact that the tool used to estimate the quantity $\norma S_{n}^{\chi}(f)-f \norma_{p}$, for $f \in W_{p}^{1}(I)$, in the first case does not apply to  the second. More in detail, in the proof of case $i)$, a crucial role is played by the useful Hardy-Littlewood maximal inequality (see, e.g., \cite{costarelli2024asymptotic}):
\begin{equation}
\label{HLMinequality}
\| M(f;\cdot) \|_{p} \le C_{p} \| f \|_{p}, \ f \in L^{p}(I), \ 1<p<+\infty,
\end{equation}
where:
\begin{equation*}
M(f;x):=\sup_{t \in I, t \ne x} \frac{1}{\assolutol t-x \assolutor} \int_{x}^{t}\assolutol f(u) \assolutor \ du
\end{equation*}
denotes the so-called Hardy-Littlewood maximal function. It is well-known that inequality (\ref*{HLMinequality}) does not hold in the case $ii)$. Now, we can prove what follows:
\begin{teorema}
\label{StimaSobolevnoncompatto}
Let $\X$ be a kernel satisfying $(\X1)$ with $\beta \ge 1$ and let $1 < p < +\infty$ be fixed. Then for any non-negative $f \in W_{p}^{1}(I)$, it turns out that: 
\begin{equation*}
\norma \op(f) - f \norma_{p} \le \frac{m_{1}(\X)}{\aconchi} n^{-1} C_{p} \norma f' \norma_{p}, \quad n \in \N^{+}, 
\end{equation*} 
where $m_{1}(\X) < +\infty$ by Lemma \ref{lemma1.2} and $C_{p}$ is the constant arising from inequality (\ref{HLMinequality}). 
\end{teorema}
\begin{proof}
In order to estimate the quantity $\norma S_{n}^{\chi}(f)-f \norma_{p}$, in what follows, for every fixed $x \in I$, let us consider the auxiliary function $f_{x}: I \rightarrow \R_{0}^{+}$, $f_{x}(t):=f(x)$, $t \in I$. Further, by $\textbf{1}:I\rightarrow \mathbb{R}$, $\textbf{1}(x):=1$, $x \in I$, we denote the identity function. Using the properties $(iii)$ and $(iv)$ of $\op$ established in Lemma \ref{lemma1.4}, together with the fact that $\op(\textbf{1})=\textbf{1}$, we have: 
\begin{equation*}
\begin{split}
\assolutol \op(f)(x)-f(x)\assolutor&=\assolutol \op(f)(x)-f(x)\op(\textbf{1})(x) \assolutor \\
& = \assolutol \op(f)(x)-\op(f_{x})(x) \assolutor \\
& \le \op(\assolutol f-f_{x} \assolutor)(x) \\
&=\frac{\displaystyle \V \left \assolutol f\left(\frac{k}{n}\right)-f(x) \right \assolutor \X(nx-k)}{\displaystyle \V \X(nx-k) },
\end{split}
\end{equation*}
for all $n \in \N^{+}$. Now, since $f \in W_{p}^{1}(I)$, with $1 < p < +\infty$, we can exploit in the previous expression the following first order Taylor formula with integral remainder (see \cite{devore1993constructive}, p. 37):
\begin{equation}
\label{Taylor}
f(u)=f(x)+\int_{x}^{u} f'(t) \ dt, \quad x, u \in I.
\end{equation}
Thus, using (\ref{Taylor}) and inequality (\ref{disuguaglianzaaconchi}), for every $n \in \N^{+}$, we can write what follows:
\begin{equation*}
\begin{split}
\integrale \assolutol \op(f)(x)-f(x) \assolutor^{p} \ dx & \le \integrale \left\{\frac{\displaystyle \V \left \assolutol \int_{x}^{k/n} f'(t) \ dt \right \assolutor \X(nx-k)}{\displaystyle \V \X(nx-k)}\right\}^{p} dx \\
& \le \frac{1}{\aconchi^{p}} \integrale \left\{\V \left \assolutol \int_{x}^{k/n} f'(t) \ dt\right \assolutor \assolutol \X(nx-k) \assolutor\right\}^{p} dx \\
&= \frac{1}{\aconchi^{p}} \integrale  \V \left \assolutol \int_{x}^{k/n} f'(t) \ dt \right \assolutor^{p} \assolutol \X(nx-k) \assolutor^{p} dx \\
&\le \frac{1}{\aconchi^{p}} \integrale  \V \left \assolutol \frac{\assolutol k/n-x \assolutor}{\assolutol k/n-x \assolutor} \int_{x}^{k/n} \assolutol f'(t) \assolutor \ dt \right \assolutor^{p} \assolutol \chi(nx-k) \assolutor^{p}  dx.
\end{split}
\end{equation*}
Now, recalling the definition of the Hardy-Littlewood maximal function, the notion of the first order generalized absolute moment of $\X$, and inequality (\ref{HLMinequality}), we finally obtain: 
\begin{equation*}
\begin{split}
\integrale \assolutol \op(f)(x)-f(x)\assolutor^{p} \ dx &\le \frac{1}{\aconchi^{p}} \integrale  \V \left \assolutol \frac{\assolutol k/n-x \assolutor}{\assolutol k/n-x \assolutor} \int_{x}^{k/n} \assolutol f'(t) \assolutor \ dt \right \assolutor^{p} \assolutol \chi(nx-k) \assolutor^{p}  dx \\
&\le \frac{1}{\aconchi^{p}} \integrale \V \left \assolutol \left\assolutol \frac{k}{n}-x\right\assolutor M(f';x) \right \assolutor^{p} \assolutol \X(nx-k) \assolutor^{p} dx \\
&= \frac{1}{\aconchi^{p}} n^{-p} \integrale \assolutol M(f';x) \assolutor^{p} \V \assolutol nx-k \assolutor^{p} \assolutol \X(nx-k) \assolutor^{p} \ dx \\ 
& \le \left(\frac{m_{1}(\chi)}{\aconchi}\right)^{p}  n^{-p} \norma M(f';\cdot) \norma_{p}^{p} \\
&\le \left(\frac{m_{1}(\chi)}{\aconchi}\right)^{p} n^{-p} C_{p}^{p}\norma f' \norma_{p}^{p}<+\infty,
\end{split}
\end{equation*}
for each $n \in \N^{+}$, where $m_{1}(\chi)<+\infty$ by Lemma \ref{lemma1.2}, since condition $(\chi1)$ is satisfied for $\beta \ge 1$. Moreover, $\norma f' \norma_{p}^{p}$ is finite, since $f$ belongs to $W_{p}^{1}(I)$. This completes the proof.
\end{proof}
In order to establish a similar estimate for case $(ii)$ $p=1$, it is necessary to require that the kernel function $\chi$ is compactly supported. Indeed, under this additional assumption, we are able to estimate the $L^{1}$-norm of the approximation error for the operators $S_{n}^{\chi}(f)$ with $f \in W_{1}^{1}(I)$ as follows.
\begin{teorema}
\label{stimaSobolevcompatto}
Let $\X$ be a kernel such that $\text{supp}(\X) \subseteq [-T, T]$, $T>0$. Then for any non-negative $f \in W_{1}^{1}(I)$, it turns out that: 
\begin{equation*}
\norma \op(f) - f \norma_{1} \le \frac{2TM_{0}(\chi)}{a_{\chi}}n^{-1}\norma f' \norma_{1}, \quad n \in \N^{+},
\end{equation*}
where $M_{0}(\chi):=\sup_{x \in \R} \sum_{k \in \Z} \assolutol \chi(x-k) \assolutor$ is finite, since $\chi$ is bounded and with compact support.
\end{teorema}
\begin{proof}
Let $n \in \N^{+}$ and $f \in W_{1}^{1}(I)$ be fixed. Proceeding as in the proof of Theorem \ref{StimaSobolevnoncompatto}, we obtain:
\begin{equation*}
\begin{split}
\integrale \assolutol \op(f)(x) - f(x) \assolutor \ dx &\le \frac{1}{\aconchi} \integrale \V \left \assolutol  \int_{x}^{k/n}  f'(t) \ dt \right \assolutor \assolutol \X(nx-k) \assolutor \ dx \\
&\le \frac{1}{\aconchi} \integrale \sum_{k=-n}^{n}  \left \assolutol \int_{x}^{k/n} \assolutol f'(t) \assolutor  \ dt \right \assolutor  \assolutol \X(nx-k) \assolutor \ dx.
\end{split}
\end{equation*}
Now, we formally extend $f'$ on the whole $\R$, as a periodic function of period $2$. Thus, recalling that $\text{supp}(\X) \subseteq [-T, T]$, and using the change of variable $y=t-x$, we get:
\begin{equation*}
\begin{split}
\integrale \assolutol \op(f)(x) -f(x) \assolutor \ dx &\le \frac{1}{a_{\chi}} \integrale \sum_{\overset{k=-n}{\assolutol nx-k \assolutor \le T}}^{n} \left \assolutol \int_{0}^{k/n-x} \assolutol f'(y+x) \assolutor \ dy  \right \assolutor \assolutol \X(nx-k) \assolutor \ dx \\ 
& \le  \frac{1}{\aconchi} \integrale \sum_{\overset{k=-n}{\assolutol nx-k \assolutor \le T}}^{n} \left[ \int_{\assolutol y \assolutor \le \assolutol k/n - x \assolutor} \assolutol f'(y+x) \assolutor \ dy \right] \assolutol \X(nx-k) \assolutor \ dx \\ 
& \le \frac{1}{\aconchi} \integrale \sum_{\overset{k=-n}{\assolutol nx-k \assolutor \le T}}^{n} \left[\int_{\assolutol y \assolutor \le T/n} \assolutol f'(y+x) \assolutor \ dy  \right] \assolutol \X(nx-k) \assolutor \ dx. \\ 
\end{split}
\end{equation*} 
We can interchange the integrals in the above computations by the Fubini-Tonelli theorem, obtaining:
\begin{equation*}
\begin{split}
\integrale \assolutol \op(f)(x)-f(x) \assolutor \ dx &\le \frac{1}{\aconchi} \integrale \sum_{\overset{k=-n}{\assolutol nx-k \assolutor \le T}}^{n} \left[\int_{\assolutol y \assolutor \le T/n} \assolutol f'(y+x) \assolutor \ dy  \right] \assolutol \X(nx-k) \assolutor \ dx \\
& = \frac{1}{\aconchi} \int_{\assolutol y \assolutor \le T/n} dy \integrale \assolutol f'(y+x) \assolutor \sum_{\overset{k=-n}{\assolutol nx-k \assolutor \le T}}^{n} \assolutol \X(nx-k) \assolutor \ dx \\
& \le \frac{M_{0}(\X)}{\aconchi} \int_{\assolutol y \assolutor \le T/n} \norma f'(y+\cdot) \norma_{1} \ dy,
\end{split}
\end{equation*}
where $M_{0}(\chi)=\sup_{x \in \R} \sum_{k \in \Z} \assolutol \X(x-k) \assolutor < +\infty$, since $\X$ is bounded and compactly supported. Finally, observing that: 
\begin{equation*}
\norma f'(y+\cdot) \norma_{1} = \norma f' \norma_{1},
\end{equation*}
for every $y \in [-T/n, T/n]$, we get the thesis.
\end{proof}
We are now ready to prove the main result of this section: a quantitative Jackson-type estimate, in terms of the first-order $\tau$-modulus, for the approximation (in the $L^{p}$-norm) of not-necessarily continuous (non-negative) functions by the max-product generalized sampling operators defined in (\ref{operatori[-1,1]}).
\begin{teorema}
\label{teorema2.3}
Let $\X \in L^{1}(\R)$ be a given kernel satisfying $(\chi1)$ with $\beta \ge 1$ and let $1 < p < +\infty$ be fixed. Then for any non-negative and bounded function $f \in L^{p}(I)$, it turns out that:
\begin{equation*}
\norma \op(f) - f \norma_{p} \le K \tau_{1}\left(f, \frac{1}{n}\right)_{p},
\end{equation*}  
for every $n \in \N^{+}$, where:
\begin{equation*}
K:= \frac{2^{(1/p)+8}((m_{0}(\X))^{p-1}\norma \X \norma_{1})^{1/p}}{\aconchi}+C_{p}c(1)\frac{m_{1}(\chi)}{\aconchi}+16,
\end{equation*}
$C_{p}$ is the constant arising from inequality (\ref{HLMinequality}) and $c(1)$ is given in $(iii)$ of Theorem \ref{teoremaSteklov} with $r=1$.
(Note that, by Lemma \ref{lemma1.2}, $m_{0}(\chi)$ and $m_{1}(\chi)$ are both finite). 
\end{teorema}
\begin{proof}
Let $n \in \N^{+}$ and $0 < h \le 2$ be fixed. Exploiting Theorem \ref{teoremaSteklov} with $r=1$, we can write what follows:
\begin{equation*}
\norma \op(f) - f \norma_{p} \le \norma \op(f) - \op(f_{1,h}) \norma_{p} + \norma \op(f_{1,h})-f_{1,h} \norma_{p} + \norma f_{1,h}-f \norma_{p},
\end{equation*}
where $f_{1,h}$ is the Steklov-type function corresponding to $f$. Now, by $(ii)$ of Theorem \ref{teoremaSteklov}, we immediately get:
\begin{equation*}
\norma f_{1,h}-f \norma_{p} \le \tau_{1}(f, 2h)_{p}.
\end{equation*}
Further, we can apply Theorem \ref{StimaSobolevnoncompatto} to $f_{1,h}$, since it is non-negative and belongs to $W_{p}^{1}(I)$, according to Theorem \ref{teoremaSteklov}. Thus, using again Theorem \ref{teoremaSteklov} $(iii)$ (with $r=1$), we obtain:
\begin{equation*}
\norma \op(f_{1,h}) -f_{1,h} \norma_{p} \le  \frac{m_{1}(\X)}{\aconchi} n^{-1} C_{p} \norma f_{1,h}' \norma_{p} \le \bar{K} n^{-1}h^{-1}\tau_{1}(f,h)_{p}, 
\end{equation*}
where: 
\begin{equation*}
\bar{K}:=\frac{m_{1}(\chi)}{\aconchi}C_{p}c(1),
\end{equation*}
for suitable positive constants $C_{p}$ and $c(1)$, and $m_{1}(\X)<+\infty$ by Lemma \ref{lemma1.2}.
\newline Now, to estimate the quantity $\norma \op(f) - \op(f_{1,h}) \norma_{p}$, we use Lemma \ref{Lipschitzproperty}, $(i)$ of (again) Theorem \ref{teoremaSteklov} (with $r=1$), Lemma \ref{Lemma1.5}, and the property $(i)$ of the first-order $\tau$-modulus, respectively, as follows:
\begin{equation*}
\begin{split}
\norma \op(f) - \op(f_{1,h}) \norma_{p} &\le \frac{((m_{0}(\X))^{p-1}\norma \X \norma_{1})^{1/p}}{\aconchi} \norma f - f_{1,h} \norma_{l^{p}({\Sigma_{n}})} \\
& =  \frac{((m_{0}(\X))^{p-1}\norma \X \norma_{1})^{1/p}}{\aconchi}  \left\{\somma \left \assolutol f\left(\frac{k}{n}\right) -f_{1,h}\left(\frac{k}{n}\right)\right \assolutor^{p} n^{-1}\right\}^{1/p} \\
& \le \frac{((m_{0}(\X))^{p-1}\norma \X \norma_{1})^{1/p}}{\aconchi} \left\{\somma \left[\omega_{1}\left(f,\frac{k}{n}, 2h\right)\right]^{p}n^{-1}\right\}^{1/p} \\
& \le \frac{2^{(1/p)+4}((m_{0}(\X))^{p-1}\norma \X \norma_{1})^{1/p}}{\aconchi}  \tau_{1}\left(f, h+\frac{1}{n+1}\right)_{p} \\
& \le  \frac{2^{(1/p)+4}((m_{0}(\X))^{p-1}\norma \X \norma_{1})^{1/p}}{\aconchi}  \tau_{1}\left(f, h+\frac{1}{n}\right)_{p}.
\end{split}
\end{equation*}
Finally, setting $h=1/n \le 1$, rearranging all the above estimates, by the property $(v)$ of $\tau_{1}(f, \cdot)_{p}$, we immediately get:
\begin{equation*}
\begin{split}
\norma \op(f) - f \norma_{p} & \le \left(\frac{2^{(1/p)+4}((m_{0}(\X))^{p-1}\norma \X \norma_{1})^{1/p}}{\aconchi} +1\right)\tau_{1}\left(f, \frac{2}{n}\right)_{p} + \bar{K}\tau_{1}\left(f, \frac{1}{n}\right)_{p} \\
& \le \left(\frac{2^{(1/p)+8}((m_{0}(\X))^{p-1}\norma \X \norma_{1})^{1/p}}{\aconchi} +16+\bar{K}\right) \tau_{1}\left(f, \frac{1}{n}\right)_{p}.
\end{split}
\end{equation*} 
This completes the proof.
\end{proof}
Note that the quantitative estimate established in Theorem \ref{StimaSobolevnoncompatto} is shown to be crucial in the proof of the previous theorem, which therefore does not hold in the case $(ii)$ $p=1$. However, in view of Theorem \ref{stimaSobolevcompatto}, we can obtain an analogous result in this case if we consider the operators $S_{n}^{\chi}$ in (\ref{operatori[-1,1]}) based upon kernels with compact support. Namely, we can state the following:
\begin{teorema}
\label{teorema2.4}
Let $\chi$ be a kernel such that $supp(\chi) \subseteq [-T, T]$, $T>0$. Then for any non-negative and bounded function $f \in L^{1}(I)$, we have:
\begin{equation*}
\norma \op(f) - f \norma_{1} \le K \tau_{1}\left(f, \frac{1}{n}\right)_{1},
\end{equation*}
for every $n \in \N^{+}$, where:
\begin{equation*}
K:=\frac{2^{9}\norma \chi \norma_{1}}{\aconchi}+c(1)\frac{2TM_{0}(\X)}{\aconchi}+16,
\end{equation*}
and $c(1)$ is (again) the constant arising from Theorem \ref{teoremaSteklov} $(iii)$ with $r=1$. 
\end{teorema}
The proof of Theorem \ref{teorema2.4} is analogous to that of Theorem \ref{teorema2.3}, except that the inequality proved in Theorem \ref{stimaSobolevcompatto} is used instead of the one from Theorem \ref{StimaSobolevnoncompatto}.\newline  As a consequence of the quantitative estimates established in Theorems \ref{teorema2.3} and \ref{teorema2.4}, combined with the convergence property of $\tau_{1}(f, \cdot)_{p}$ recalled in Proposition \ref{proposizione1.1} (with $r=1$), the $L^{p}$-convergence for the max-product generalized sampling operators in (\ref{operatori[-1,1]}) to any non-negative function $f$ immediately follows, provided that $f$ is bounded, measurable, and Riemann integrable on the interval $I$.
\begin{teorema}
\label{risultatoconvergenza}
Under the same assumptions of Theorem \ref{teorema2.3}, for any non-negative Riemann integrable function $f \in M(I)$, we have:
\begin{equation*}
\lim_{n \rightarrow +\infty} \norma S_{n}^{\chi}(f) - f \norma_{p} = 0,  \ \ 1 < p < +\infty. 
\end{equation*} 
Further, under the same assumptions of Theorem \ref{teorema2.4}, for every non-negative Riemann integrable function $f \in M(I)$, we get:
\begin{equation*}
\lim_{n \rightarrow +\infty} \norma S_{n}^{\chi}(f)-f \norma_{1} =0.
\end{equation*} 
\end{teorema}
\begin{remark}
$(a)$ Note that all the results proved in this section still hold if we consider the approximation of not-necessarily non-negative functions $f:I \rightarrow \R$, as long as they are bounded from below. Indeed, to prove this, it is sufficient to take into account the non-linear operators $(L_{n}^{\chi}(f))_{n \in \N^{+}}$ defined by $L_{n}^{\chi}(f)(x):=S_{n}^{\chi}(f-c)(x)+c$, $x \in I$, where $c:=\inf_{x \in I}f(x)$. 
\newline $(b)$ The quantitative estimates established in Theorems \ref{teorema2.3} and \ref{teorema2.4} (and, hence, also the convergence result in Theorem \ref{risultatoconvergenza}) for the max-product generalized sampling operators can be applied to several kernels from the wide family of functions known in the literature (see, for instance, \cite{coroianu2019max,coroianu2021approximation}), which satisfy the (weak) assumptions $(\chi1)$ and $(\chi2)$.  
\end{remark}
\section{Shape-preserving properties for the max-product generalized sampling operators \label{Sezione4}}
In this section, we present some shape-preserving properties of the max-product generalized sampling operators defined in Definition \ref{definizione1.1}. First, we prove that, under suitable additional assumptions on the kernel $\chi$, the non-linear operators $S_{n}^{\chi}(f)$ partially preserve monotonicity on the interval $[0,1]$, whenever $f:[0,1] \rightarrow \R_{0}^{+}$ is a bounded function that is either non-increasing or non-decreasing on the entire interval. The approach we follow is based on the idea originally developed by Coroianu and Gal in \cite{coroianu2011approximation}, where the authors established a similar result only for the sinc and Fejér kernels. \newline In order to achieve the same result for a wider class of max-product generalized sampling operators, we require that the kernel $\chi$ from Definition \ref{definizione1.1} satisfies the following additional assumptions: 
\vskip0.2cm $(\chi3)$ $\chi \in C^{1}(\R)$, $\chi(x) \ge 0$ for every $x \in \R$, and moreover $\lim_{\assolutol x \assolutor \rightarrow +\infty} \chi(x)=0$;
\vskip0.05cm $(\chi4)$ $\chi(x)$ is an even function;
\vskip0.05cm $(\chi5)$ $\chi(x)$ is non-decreasing for $x <0$ and non-increasing for $x \ge 0$.
\begin{remark}
\label{Remark3.1}
Note that conditions $(\chi4)$ and $(\chi5)$ imply that $\chi(1)=a_{\chi}>0$, where $a_{\chi}$ is the constant of condition $(\chi2)$.
\end{remark}
Under the assumptions $(\chi1)-(\chi5)$, we study in this section the shape-preserving properties of the following family of max-product operators: 
\begin{equation}
\label{operatori[0,1]}
S_{n}^{\chi}(f)(x)=\frac{\displaystyle \bigvee_{k=0}^{n}f\left(\frac{k}{n}\right)\chi(nx-k)}{\displaystyle \bigvee_{k=0}^{n}\chi(nx-k)} , \ \ x \in [0,1], \ n \in \N^{+},
\end{equation}
where $f:[0,1] \rightarrow \R_{0}^{+}$ is bounded. 

\begin{remark}
$(a)$ Since, in this setting, Lemma \ref{disuguaglianzacruciale} holds with $a_{\chi}=\chi(1)>0$, we know that $S_{n}^{\chi}(f)(x)$  is a well-defined function for all $x \in [0,1]$ and $n \in \N^{+}$. Obviously, condition $(\chi3)$ implies that both the numerator and denominator in the right-hand side of (\ref{operatori[0,1]}) are continuous functions, as they are the maximum of a finite number of continuous functions. Thus, we immediately obtain the continuity of $S_{n}^{\chi}(f)(x)$ for all $x \in [0,1]$. 
\newline $(b)$ As noted in Remark \ref{remark2.1} too, if we replace condition $(\chi2)$ with the slightly weaker assumption $(\overline{\chi2})$ while keeping all other conditions, the main results of this section (see Theorem \ref{TeoremaSPP}) still hold. Furthermore, as observed in Remark \ref{Remark3.1}, in this case we would have $\chi(\frac{1}{2})=a_{\chi}>0$. 
\end{remark}
In order to establish the main results of this section, we need the following lemma.
\begin{lemma}
\label{LemmaSPP}
Let $\chi$ be a given kernel satisfying $(\chi1)-(\chi5)$. For any $j \in \{0, 1 ,..., n-1\}$, $n \in \N^{+}$, we have: 
\begin{equation*}
\begin{split}
&\bigvee_{k=0}^{n} \chi(nx-k) = \chi(nx-j), \ for \ every \ x \in \left[\frac{j}{n}, \frac{j}{n}+\frac{1}{2n}\right], \\ 
&\bigvee_{k=0}^{n} \chi(nx-k)=\chi(nx-(j+1)), \ for \ every \ x \in \left[\frac{j}{n}+\frac{1}{2n}, \frac{j}{n}+\frac{1}{n}\right].
\end{split}
\end{equation*}
\end{lemma}
\begin{proof}
Let $j \in \{0, 1, ..., n-1\}$, $n \in \N^{+}$, and $x \in \left[\frac{j}{n}, \frac{j}{n}+\frac{1}{2n}\right]$ be fixed. Then we have $0 \le nx-j \le \frac{1}{2}$ which implies
\begin{equation*}
\chi(nx-j) \ge \chi\left(\frac{1}{2}\right) \ge \chi(1)>0,
\end{equation*} 
since $\chi(x)$ is non-increasing for $x \ge 0$ by $(\chi5)$ and $(\chi2)$ holds with $a_{\chi}=\chi(1)$. If $k \in \{0, 1,..., n\}$, with $k \ne j$, we can write what follows:
\begin{equation*}
\assolutol nx-k \assolutor =n \biggl\assolutol x -\frac{k}{n} \biggr\assolutor \ge n \cdot \frac{1}{2n} =\frac{1}{2}.
\end{equation*}
Now, using assumptions $(\chi4)$ and $(\chi5)$ (again), we get:
\begin{equation*}
\chi(nx-k)=\chi(\assolutol nx-k \assolutor) \le \chi\left(\frac{1}{2}\right),
\end{equation*} 
and the first part of the thesis immediately follows. 
\newline  Let now $x \in \left[\frac{j}{n}+\frac{1}{2n}, \frac{j}{n}+\frac{1}{n}\right]$ be fixed. This implies $-\frac{1}{2} \le nx-(j+1) \le 0$, thus, by using again the fact that $\chi(x)$ is even and non-decreasing for $x <0$, we have: 
\begin{equation*}
\chi(nx-(j+1)) \ge \chi\left(\frac{1}{2}\right) \ge \chi(1)>0.
\end{equation*} 
Observing that for every $k \in \{0,1,...,n\}$, with $k \ne j+1$,
\begin{equation*}
\assolutol nx-k \assolutor= n \biggl\assolutol x-\frac{k}{n} \biggr\assolutor \ge \frac{1}{2}, 
\end{equation*} 
we finally obtain:
\begin{equation*}
\chi(nx-k)=\chi(\assolutol nx-k \assolutor) \le \chi\left(\frac{1}{2}\right). 
\end{equation*}
This completes the proof.
\end{proof}
Now, we are able to prove that the max-product generalized sampling operator $S_{n}^{\chi}(f)$ in (\ref{operatori[0,1]}) partially preserves the monotonicity of $f$ on $[0,1]$ for all $n \in \N^{+}$. 
\begin{teorema}
\label{TeoremaSPP}
Under the assumptions on $\chi$ of Lemma \ref{LemmaSPP}, we have:
\newline $(i)$ If $f: [0,1] \rightarrow \R^{+}_{0}$ is  bounded and non-decreasing on $[0,1]$ then for every $n \in \N^{+}$, the max-product generalized sampling operator $S_{n}^{\chi}(f)$ is non-decreasing on each sub-interval $[\frac{j}{n},\frac{j}{n}+\frac{1}{2n}]$, $j \in \{0,1,...,n-1\}$. \\
$(ii)$ If $f:[0,1] \rightarrow \R^{+}_{0}$ is bounded and non-increasing on $[0,1]$ then for every $n \in \N^{+}$, the max-product generalized sampling operator $S_{n}^{\chi}(f)$ is non-increasing on each sub-interval $\left[\frac{j}{n}+\frac{1}{2n}, \frac{j}{n}+\frac{1}{n}\right]$, $j \in \{0,1,...,n-1\}$. 
\end{teorema}
\begin{proof}
By using Lemma \ref{LemmaSPP}, we can write $S_{n}^{\chi}(f)(x)$ as follows:
\begin{equation*}
S_{n}^{\chi}(f)(x)=\bigvee_{k=0}^{n} \frac{\chi(nx-k)}{\chi(nx-j)}  \cdot f\left(\frac{k}{n}\right), \ x \in \left[\frac{j}{n}, \frac{j}{n}+\frac{1}{2n}\right],
\end{equation*}
and 
\begin{equation*}
S_{n}^{\chi}(f)(x)=\bigvee_{k=0}^{n} \frac{\chi(nx-k)}{\chi(nx-(j+1))} \cdot f\left(\frac{k}{n}\right), \ x \in \left[\frac{j}{n}+\frac{1}{2n}, \frac{j}{n}+\frac{1}{n}\right],
\end{equation*}
for any $j \in \{0,1,...,n-1\}$, $n \in \N^{+}$. \newline $(i)$ Let $f:[0,1] \rightarrow \R^{+}_{0}$ be a non-decreasing function on $[0,1]$ and let $j \in \{0,1, ..., n-1\}$, $n \in \N^{+}$, be fixed. If $x \in [\frac{j}{n}, \frac{j}{n} + \frac{1}{2n}]$ then, by Lemma \ref{LemmaSPP}, we have:
\begin{equation*}
\frac{\chi(nx-k)}{\chi(nx-j)} \cdot f\left(\frac{k}{n}\right) \le f\left(\frac{k}{n}\right) \le f\left(\frac{j}{n}\right),
\end{equation*} 
for all $k \le j$, which implies that: 
\begin{equation*}
\begin{split}
S_{n}^{\chi}(f)(x)&= \max\left\{\bigvee_{k=0}^{j}\frac{\chi(nx-k)}{\chi(nx-j)}\cdot f\left(\frac{k}{n}\right), \bigvee_{k=j+1}^{n}\frac{\chi(nx-k)}{\chi(nx-j)}\cdot f\left(\frac{k}{n}\right)\right\} \\
&=\max\left\{f\left(\frac{j}{n}\right), \bigvee_{k=j+1}^{n} \frac{\chi(nx-k)}{\chi(nx-j)} \cdot f\left(\frac{k}{n}\right) \right\}, \ x \in \left[\frac{j}{n}, \frac{j}{n}+\frac{1}{2n}\right].
\end{split}
\end{equation*}
Now, for each $k >j$ let us consider the function
\begin{equation*}
g_{k,n,j}: \left[\frac{j}{n}, \frac{j}{n} +\frac{1}{2n}\right] \rightarrow \R^{+}_{0}
\end{equation*}
\begin{equation*}
g_{k,n,j}(x):=\frac{\chi(nx-k)}{\chi(nx-j)}\cdot f\left(\frac{k}{n}\right).
\end{equation*}
By computing the first derivative of $g_{k,n,j}(x)$, which exists since $\chi \in C^{1}(\R)$ by $(\chi3)$, we obtain:
\begin{equation*}
g'_{k,n,j}(x)=\frac{n[\chi'(nx-k)\chi(nx-j)-\chi'(nx-j)\chi(nx-k)]}{(\chi(nx-j))^{2}} \cdot f\left(\frac{k}{n}\right),  \ x \in \left[\frac{j}{n}, \frac{j}{n}+\frac{1}{2n}\right]. 
\end{equation*}
If $k \ge j+1$ then $nx-k \le j-k +\frac{1}{2} \le -\frac{1}{2}$ and $nx-j \ge 0$ for all $x \in  [\frac{j}{n}, \frac{j}{n}+\frac{1}{2n}]$, thus, by the property $(\chi5)$ and the fact that $\chi(x)$ is non-negative by $(\chi3)$, we have $\chi'(nx-k)\chi(nx-j) \ge 0$ and $\chi'(nx-j)\chi(nx-k) \le 0$. Since this implies that $g'_{k,n,j}(x) \ge 0$ for every $x \in \left[\frac{j}{n}, \frac{j}{n}+\frac{1}{2n}\right]$ and $k >j$, $S_{n}^{\chi}(f)$ is non-decreasing on $[\frac{j}{n}, \frac{j}{n}+\frac{1}{2n}]$ as a maximum of non-decreasing functions.
\newline $(ii)$ Now, we suppose that $f:[0,1] \rightarrow \R^{+}_{0}$ is a non-increasing function on $[0,1]$. Let $j \in \{0,1,..., n-1\}$, $n \in \N^{+}$, be fixed. If $x \in \left[\frac{j}{n}+\frac{1}{2n}, \frac{j}{n}+\frac{1}{n}\right]$, by using again Lemma \ref{LemmaSPP}, we can write what follows: 
\begin{equation*}
\frac{\chi(nx-k)}{\chi(nx-(j+1))} \cdot f\left(\frac{k}{n}\right) \le f\left(\frac{k}{n}\right) \le f\left(\frac{j+1}{n}\right),
\end{equation*}   
for all $k>j$. Therefore, we immediately get: 
\begin{equation*}
\begin{split}
S_{n}^{\chi}(f)(x)&= \max\left\{\bigvee_{k=0}^{j}\frac{\chi(nx-k)}{\chi(nx-(j+1))}\cdot f\left(\frac{k}{n}\right), \bigvee_{k=j+1}^{n}\frac{\chi(nx-k)}{\chi(nx-(j+1))}\cdot f\left(\frac{k}{n}\right)\right\} \\
& = \max\left\{ \bigvee_{k=0}^{j}\frac{\chi(nx-k)}{\chi(nx-(j+1))}\cdot f\left(\frac{k}{n}\right), f\left(\frac{j+1}{n}\right)\right\}, \ x \in \left[\frac{j}{n}+\frac{1}{2n}, \  \frac{j}{n}+\frac{1}{n}\right].
\end{split}
\end{equation*}
Now, for each $k \le j$ we consider the function
\begin{equation*}
h_{k,n,j}: \left[\frac{j}{n}+\frac{1}{2n}, \frac{j}{n}+\frac{1}{n}\right] \rightarrow \R^{+}_{0}
\end{equation*}
\begin{equation*}
h_{k,n,j}(x):=\frac{\chi(nx-k)}{\chi(nx-(j+1))}\cdot f\left(\frac{k}{n}\right). 
\end{equation*}
We have: 
\begin{equation*}
h_{k,n,j}'(x)=\frac{n[\chi'(nx-k)\chi(nx-(j+1))-\chi'(nx-(j+1))\chi(nx-k)]}{(\chi(nx-(j+1)))^{2}} \cdot f\left(\frac{k}{n}\right), \ x \in \left[\frac{j}{n}+\frac{1}{2n}, \frac{j}{n}+\frac{1}{n}\right]. 
\end{equation*}
If $k \le j$ then $nx-k \ge j-k+\frac{1}{2} \ge \frac{1}{2}$ and $ nx-(j+1) \le 0$ for all $x \in \left[\frac{j}{n}+\frac{1}{2n}, \frac{j}{n}+\frac{1}{n}\right]$, thus, we get $\chi'(nx-k)\chi(nx-(j+1)) \le 0$ and $\chi'(nx-(j+1))\chi(nx-k) \ge 0$, since $\chi(x)$ is non-increasing for $x \ge 0$ and non-decreasing for $x <0$ by $(\chi5)$, respectively, together with assumption $\chi(x) \ge 0$ for every $x \in \R$ of $(\chi3)$. This implies that $h_{k,n,j}'(x) \le 0$ for all $x \in \left[\frac{j}{n}+\frac{1}{2n} \frac{j}{n}+\frac{1}{n}\right]$ and $k \le j$, hence, $S_{n}^{\chi}(f)$ is non-increasing on $\left[\frac{j}{n}+\frac{1}{2n}, \frac{j}{n}+\frac{1}{n}\right]$ as a maximum of non-increasing functions. This completes the proof.     
\end{proof}
\begin{remark}
$(a)$ Suppose that $f:[0,1]\rightarrow \mathbb{R}_{0}^{+}$ is a bounded and non-decreasing function on $[0,1]$, and let $j \in \{0,1,\dots,n-2\}$, with $n \in \mathbb{N}^{+}$, be fixed. Reasoning as in the proof of Theorem \ref{TeoremaSPP} $(i)$, we have:
\begin{equation*}
	S_{n}^{\chi}(f)(x)=\max\left\{f\left(\frac{j+1}{n}\right),\bigvee_{k=j+2}^{n}\frac{\chi(nx-k)}{\chi(nx-(j+1))}\cdot f\left(\frac{k}{n}\right)\right\},
\end{equation*}  
for every $x \in \left[\frac{j}{n}+\frac{1}{2n},\frac{j}{n}+\frac{1}{n}\right]$. Now, setting 
\begin{equation*}
	h_{k,n,j}(x):=\frac{\chi(nx-k)}{\chi(nx-(j+1))}\cdot f\left(\frac{k}{n}\right), \quad  x \in \left[\frac{j}{n}+\frac{1}{2n},\frac{j}{n}+\frac{1}{n}\right],
\end{equation*}
with $k>j+1$, it is not possible to determine the sign of $h'_{k,n,j}(x)$ under assumptions $(\chi1)-(\chi5)$ on the kernel. It follows that the monotonicity of $h_{k,n,j}(x)$, for $k>j+1$, and consequently that of $S_{n}^{\chi}(f)(x)$, remains unknown on each interval $ \left[\frac{j}{n}+\frac{1}{2n},\frac{j}{n}+\frac{1}{n}\right]$, for all $j \in \{0,1,\dots,n-2\}$, with $n \in \mathbb{N}^{+}$, whenever $\chi$ is a general kernel function. \newline Similarly, let $f:[0,1]\rightarrow\mathbb{R}_{0}^{+}$ be a bounded and non-increasing function, and let $j \in \{1,2,\dots,n-1\}$, with $n \in \mathbb{N}^{+}$, be fixed. Repeating the same arguments as in the proof of Theorem \ref{TeoremaSPP} $(ii)$, we obtain:
\begin{equation*}
	S_{n}^{\chi}(f)(x)=\max\left\{\bigvee_{k=0}^{j-1} \frac{\chi(nx-k)}{\chi(nx-j)} \cdot f\left(\frac{k}{n}\right), f\left(\frac{j}{n}\right) \right\},
\end{equation*}  
for every $x \in \left[\frac{j}{n}, \frac{j}{n} +\frac{1}{2n}\right]$. Now, defining 
\begin{equation*}
	g_{k,n,j}(x):=\frac{\chi(nx-k)}{\chi(nx-j)}\cdot f\left(\frac{k}{n}\right), \quad x \in \left[\frac{j}{n}, \frac{j}{n} +\frac{1}{2n}\right],
\end{equation*}
with $k<j$, it is not possible to determine the sign of $g'_{k,n,j}(x)$ whenever $\chi$ is a general kernel function satisfying assumptions $(\chi1)-(\chi5)$. Therefore, the monotonicity of $g_{k,n,j}(x)$, for $k<j$, and consequently that of $S_{n}^{\chi}(f)(x)$, cannot be established on each interval $\left[\frac{j}{n}, \frac{j}{n} +\frac{1}{2n}\right]$, for all $j \in \{1,2,\dots,n-1\}$, with $n \in \mathbb{N}^{+}$, without specifying the explicit form of the kernel.  \\
$(b)$ In all the above results, we considered kernel functions satisfying the assumptions $(\chi i)$, $i=1,2,3,4,5$. In particular, the shape-preserving properties established in the previous theorem are mainly based on the properties $(\chi3)-(\chi5)$, as these are crucial in the proof of both Lemma \ref{LemmaSPP} and Theorem \ref{TeoremaSPP}. A classic class of functions that satisfy the above assumptions are the (smooth) centered bell-shaped functions, as defined in \cite{cardaliaguet1992approximation} by Cardaliaguet and Euvrard. \newline The archetypal (centered) bell-shaped function is the Gaussian function  $\chi_{g}(x):=ae^{-x^{2}/2b^{2}}$, $a$, $b >0$, i.e., the probability density function of a normally distributed random variable with zero expected value. Another typical example of a centered bell-shaped function is the well-known hyperbolic secant $\chi_{h}(x):=sech(x)=\frac{2}{e^{x}+e^{-x}}$. Clearly, both $\chi_{g}$ and $\chi_{h}$ are smooth (non-negative) functions, and they satisfy all the assumptions $(\chi i)$ for $i=1,2,3,4,5$. \newline In order to find other examples of kernels to which the above theory can be applied, it is useful to recall an interesting property of bell-shaped functions: namely, the primitive of any bell-shaped function is a sigmoidal function. Therefore, for example, the derivative of the logistic function $\chi_{l}(x):= e^{-x}/(1+e^{-x})^{2}$ and the derivative of the arctangent function $\chi_{a}(x):=1/(1+x^{2})$ (including its scaled versions) are both bell-shaped, $C^{\infty}$, even and centered. The function defined by:
\begin{equation*}
\chi(x):=\begin{cases}
\displaystyle \frac{2}{\pi}\arctan\left(\frac{1}{x^{2}}\right), & x \ne 0, \\
1, & x=0.
\end{cases}
\end{equation*}
also satisfies assumptions $(\chi1)-(\chi5)$. 
\newline Finally, we conclude by providing an example of a duration-limited kernel, that is, one with compact support, which can be used to construct a family of max-product operators that satisfies Theorem \ref{TeoremaSPP}. We can consider the following centered bell-shaped function:
\begin{equation}
\label{centralb-spline}
\chi_{b}(x):= \begin{cases} \displaystyle \frac{3}{4}-x^2, & \displaystyle  \assolutol x\assolutor \le \frac{1}{2},\\
	\displaystyle \frac{1}{2}\left(\frac{3}{2}-\assolutol x \assolutor\right)^2, & \displaystyle \frac{1}{2} < \assolutol x \assolutor \le \frac{3}{2},\\
\displaystyle 	0, & \displaystyle \assolutol x \assolutor > \frac{3}{2},
\end{cases}
\end{equation}  
which is the so-called central B-spline of order $3$. Note that $\chi_{b}(x)$ is $C^{1}$, even and with $\text{supp}(\chi_{b})\subset[-\frac{3}{2},\frac{3}{2}]$. 
\end{remark}

\section*{Acknowledgments}

{\small The authors are members of the Gruppo Nazionale per l'Analisi Matematica, la Probabilit\`a e le loro Applicazioni (GNAMPA) of the Istituto Nazionale di Alta Matematica (INdAM), of the Gruppo UMI (Unione Matematica Italiana) T.A.A. (Teoria dell'Approssimazione e Applicazioni)}, and of the network RITA (Research ITalian network on Approximation).


\section*{Conflict of interest/Competing interests}

{\small The authors declare that they have no conflict of interest and competing interest.}

\section*{Availability of data and material and Code availability}

{ \small Not applicable.}


\end{document}